\documentclass[a4paper,12pt]{article}

\usepackage{amsmath,amssymb,amsthm}
\usepackage{bbm, mathtools}
\usepackage{color}
\usepackage[numbers]{natbib} 

\newtheorem{theorem}{Theorem}
\newtheorem{lemma}{Lemma}
\newtheorem{corollary}{Corollary}

\newcommand{\R}{{\mathbb{R}}}

\def\={\stackrel {\rm def}  {=}}
\def\di={\overset{\text{${\mathcal{D} }$}} =}

\title{Small ball probabilities, maximum density and rearrangements}
\author{T. Ju\v skevi\v cius$^{1}$, J. D. Lee$^{2}$}
\date{}
\begin{document}
\maketitle

\footnotetext[1] {University of Memphis, Memphis, TN, USA, email - tomas.juskevicius@gmail.com.}

\footnotetext[2] {University of Cambridge, Cambridge, UK, email -  j.d.lee@dpmms.cam.ac.uk, jdlee0@gmail.com.}

\begin{abstract}
We prove that the probability that a sum of independent random variables in $\R^d$ with bounded densities lies in a ball is maximized by taking uniform distributions on balls. This in turn generalizes a result by Rogozin on the maximum density of such sums on the line.
\end{abstract}

Let $\mu$ be the Lebesgue measure on $\R^d$, and let X be a random vector in $\R^d$. If $X$ has a density $p$, we define
$$M(X)=\text{ess}\sup p := \sup\{\epsilon : \mu(\{t : p(t) > \epsilon\}) > 0\}.$$ 
For random variables with distributions that are not absolutely continuous with respect to $\mu$ measure we set $M(X):=\infty$. We note that $\text{ess}\sup$ is invariant under changes to $p$ on sets of measure $0$. Hence we will take our density functions to be equivalence classes up to alterations on sets of measure $0$; that is, they are defined as elements of $L_{\infty}$.\\

The aim of this paper is to provide best possible upper bounds for the maximum density and small ball probabilities of sums of random vectors.\\

Our starting point is a result by Rogozin, who showed that in the case $d=1$ the worst case is provided by uniform distributions over intervals. To be more precise, it was proved in \cite{R} that for independent real random variables $X_1,\ldots,X_n$ with $M(X_i)\leq M_i$ we have 
$$M(X_1+\cdots+X_n)\leq M(U_1+\cdots+U_n),$$
where $U_k$ are independent and uniformly distributed in $[-\frac{1}{2M_i},\frac{1}{2M_i}]$.\\

We extend Rogozin's inequality to all dimensions. In fact, we prove a more general statement for small ball probabilities that immediately implies a generalisation of Rogozin's result.

\begin{theorem}\label{set-rearrange}
Let $X_1,\ldots,X_n$ be independent random vectors in $\R^d$ with $M(X_i)\leq K_i$. Consider a collection of independent random vectors $U_1,\ldots,U_n$ with densities equal to $K_i$ on a ball around the origin and $0$ elsewhere. Then for every measurable set $S$ we have 
\begin{equation}\label{set-rearrange-ineq}
\mathbb{P}\left(X_1+\cdots+X_n\in S\right)\leq \mathbb{P}\left(U_1+\cdots+U_n 
\in B\right),
\end{equation}
where $B$ is the centered ball such that $\mu(B)=\mu(S).$
\end{theorem}

\begin{corollary}
Under the same conditions as above we also have that
$$M(X_1+\cdots+X_n)\leq M(U_1+\cdots+U_n).$$
\end{corollary}
\begin{proof}
Note that for any variable $X$ with density $p$
\[
M(X) = \lim_{\epsilon \rightarrow 0} \sup_{\mu(S) = \epsilon} \epsilon^{-1}\int_S p d\mu,
\]
and from Theorem~\ref{set-rearrange} for every fixed $\epsilon$ the right hand side is not decreased by taking the variables $U_i$ in place of $X_i$. Hence the corollary holds. 
\end{proof}

Even for $d=1$ our approach to Theorem~\ref{set-rearrange} is quite different than that of Rogozin, who used discretization arguments together with an idea of Erd\H{o}s to relate small ball probabilities to Sperner's theorem in finite set combinatorics. We avoid these subtleties by using a rearrangement inequality proved by Brascamp, Lieb and Luttinger.

Before stating this result, we define the spherically symmetric decreasing rearrangement. Given a non-negative function $f:\R^d\mapsto \R$ we first set $M_{y}^{f}=\left\{t: f(t)\geq y\right\}$. Suppose we are given an $f$ such that $M_{a}^f< \infty$ for some $a\in \R$. We define $\tilde{f}$ to be a function such that:\\
\begin{eqnarray*}
&&1)\,\, \tilde{f}(x)=\tilde{f}(y),\, \text{for} |x|_2=|y|_2;\\
&&2)\,\, f(x)\leq f(y)\,\,\text{for}\, x\leq y;\\
&&3)\,\, M_{y}^{\tilde{f}}=M_{y}^{f}.
\end{eqnarray*}
The function $\tilde{f}$ is known as the spherically symmetric decreasing rearrangement of $f$. For existence, uniqueness and other properties of $\tilde{f}$ we refer the reader to \cite{BLL} and \cite{Inequalities}.\\

Having introduced the relevant symmetrization we can state the aforementioned rearrangement result.

\begin{theorem}\label{BLL-rearrange}
Let $f_j$, $1\leq j \leq k$ be non-negative measurable functions on $\R^d$ and let $a_{j,m}$, $1\leq j \leq k$,$1\leq m \leq n$, be real numbers. Then
$$\int_{\R^{nd}}\prod_{j=1}^{k}\left(f_j\left(\sum_{m=1}^{n}a_{j,m}x_m\right)\right)d^{nd}\leq \int_{\R^{nd}}\prod_{j=1}^{k}\left(\tilde{f}_j\left(\sum_{m=1}^{n}a_{j,m}x_m\right)\right)d^{nd}$$

\end{theorem} 

A direct consequence of the latter result is the following.
\begin{theorem}\label{spherical-symmetry}
Let $X_1,\ldots,X_n$ be independent random variables with given density functions $p_{i}$. Consider another collection of independent random variables $X_1',\ldots, X_n'$ with density functions $\tilde{p}_i$. Then for every measurable set $S$ we have 
\begin{equation}
\mathbb{P}\left(X_1+\cdots+X_n\in S\right)\leq \mathbb{P}\left(X_1'+\cdots+X_n' 
\in B\right),
\end{equation}
where $B$ is the centered ball such that $\mu(B)=\mu(S).$
\end{theorem}
\begin{proof}
We have that
\[
\mathbb{P}\left(\sum X_i \in S\right) = \int_{x_1, \ldots x_n}\prod_{i=1}^{n} p_i(x_i) \mathbbm{1}_S\left(\sum_i x_i\right) d^n\mu.
\]
Now apply Theorem~\ref{BLL-rearrange} with the $f_i$ taken to be $\{p_1, \ldots, p_n, \mathbbm{1}_S\}$ and the $a_{j,m} = 1$ when $j = m$ or $j = n+1$ and $a_{j,m} = 0$ otherwise. We note that $\widetilde{\mathbbm{1}_S} = \mathbbm{1}_B$ and that $\tilde{p_i}$ are the densities of $X'_i$, completing the proof.
\end{proof}

To obtain Theorem 1 we will first characterize the extreme points of the set of measures with bounded densities.
\begin{lemma}\label{extreme-points}
Let $\mathcal{S}_K$ be the set of probability measures in $\R^d$ that have essential suprema bounded by $K>0$. The extreme points of $\mathcal{S}_K$ are measures having densities $p(t)=K\mathbb{I}_{S}(t)$ for some set $S$ with $\mu(S)=1/K.$
\end{lemma}

\begin{proof}
Firstly, we note that all measures having densities $p=K\mathbb{I}_{S}$ are extremal. Suppose not. Then $p = \alpha p_{1} + (1-\alpha)p_2$, where $\alpha \in (0,1)$ and $p_1, p_2$ are not equal to $p$. But then $p_1$ and $p_2$ differ from $p$ on a set of positive measure, and so $\max(p_1, p_2) > K$ on some set of positive measure. Hence one of $p_1$, $p_2$ must exceed $K$ on a set of positive measure, so is outside of $\mathcal{S}_K$.

Suppose that the density of a measure is not one of these extremal examples. Consider the sets
\[
A_y=\left\{t: p(t)\geq y\right\}.
\]
Now, there is some $y\in(0,K)$ such that $\mu(A_y) > 0$, as otherwise $p(t) = K$ almost everywhere on its support, and so $p$ would be one of our extremal examples. We fix any such $y$, and define $X = \sup(p) \backslash A_y$. Furthermore, we partition $X$ into two disjoint sets $X_1, X_2$ such that $\int_{X_1} p d\mu = \int_{X_2} p d\mu$.

We fix $\delta \in (0,K/y-1)\cap (0,1)$, and construct two densities $p_1, p_2$ as follows:
\[
p_i(t) = \begin{dcases}p(t) & t \in A_y \\ (1-\delta)p(t) & t\in X_i \\ (1+\delta)p(t) & t \in X_{1-i}\end{dcases}
\]
First, we observe that $p = \frac{1}{2}(p_1 + p_2)$. Furthermore, each of $p_1, p_2$ are equal to $p$ on $A_y$, and are bounded pointwise on $X$ by:
\[
(1+\delta)\sup_{X}p \leq (1+\delta)y \leq K.
\]
Hence the essential suprema of $p_1, p_2$ are bounded by $K$, and so $p_1, p_2 \in \mathcal{S}_K$ as required. 
\end{proof}

We now prove Theorem~\ref{set-rearrange}:
\begin{proof}
We first observe that Equation~\ref{set-rearrange-ineq} can be written as a multilinear integral over the densities of $X_i$ and the indicator function of $S$. As a corollary, it is maximized when each $p_i$ is an extremal member of $\mathcal{S}_{K_i}$.

Hence from Lemma~\ref{extreme-points} each density $p_i$ is proportional to the indicator function of a set of measure $K_i^{-1}$. From Theorem~\ref{spherical-symmetry}, we have that to maximize this expression we may replace each of the densities $p_i$ by $\tilde{p_i}$ and replace $S$ by  a ball $B$ of the same volume.

We now observe that if $p_i$ is proportional to an indicator function, then $\tilde{p_i}$ is proportional to the indicator function of a ball centered on the origin, which completes the theorem.
\end{proof}


\bibliography{Densities}
\nocite{RV}
\nocite{R}
\nocite{BLL}
\nocite{BC}

\end{document}